\documentclass[a4paper]{amsart}

\usepackage[latin1]{inputenc}

\usepackage{amssymb}
\usepackage{amsmath}
\usepackage{amsfonts}
\usepackage{amssymb}
\usepackage{amsthm}
\usepackage{enumerate}
\usepackage{hyperref}
\usepackage{cite}
\usepackage{calrsfs}

\newtheorem{theorem}{Theorem}[section]
\newtheorem{lemma}[theorem]{Lemma}
\newtheorem{proposition}[theorem]{Proposition}

\newtheorem{remark}[theorem]{Remark}

\newcommand{\R}{\mathbb{R}}
\newcommand{\N}{\mathbb{N}}

\newcommand{\Z}{\mathbb{Z}}

\newcommand{\cC}{{\mathcal C}}

\newcommand{\cF}{{\mathcal F}}

\newcommand{\cR}{{\mathcal R}}   
\newcommand{\cS}{{\mathcal S}}

\newcommand{\cW}{{\mathcal W}}

\newcommand{\mR}{{\mathbf R}}
\newcommand{\mK}{{\mathbf K}}

\newcommand{\dist}{\operatorname{dist}}
\newcommand{\supp}{\operatorname{supp}}
\newcommand{\weakto}{\rightharpoonup}

\newcommand{\embed}{\hookrightarrow}

\newcommand{\eps}{\varepsilon}

\renewcommand{\phi}{\varphi}

\def\eps{\varepsilon}

\def\wt{\widetilde}

\def\wh{\widehat}

\begin{document}
\title[The nonlinear Helmholtz equation in the plane]{
	Existence and asymptotic behavior of standing waves of the nonlinear Helmholtz equation in the plane}
\author{Gilles Ev\'equoz}
\address{Institut f\"ur Mathematik, Johann Wolfgang Goethe-Universit\"at,
Robert-Mayer-Str. 10, 60054 Frankfurt am Main, Germany}
\email{evequoz@math.uni-frankfurt.de}

\begin{abstract}
 In this paper we study the semilinear elliptic problem
 $$
 -\Delta u -k^2u=Q|u|^{p-2}u\quad\text{ in }\R^2,
 $$
 where $k>0$, $p\geq 6$ and $Q$ is a bounded function. 
 We prove the existence of real-valued $W^{2,p}$-solutions, 
 both for decaying and for periodic coefficient $Q$. 
 In addition, a nonlinear far-field relation 
 is derived for these solutions.	  	
\end{abstract}


\maketitle

\section{Introduction and main results}\label{sec:introduction}
  The purpose of this article is to study the existence and the properties of real-valued solutions 
  of the semilinear problem
  \begin{equation}\label{eq:33a}
  -\Delta u +\lambda u=Q|u|^{p-2}u\quad\text{ in }\R^2
  \end{equation}
  that vanish at infinity, in the case where $p>2$ and $Q$: $\R^2$ $\to$ $\R$ is a bounded function.
  For $\lambda\geq 0$, the problem \eqref{eq:33a} in $\R^N$ with such superlinear nonlinearities 
  has received a great deal of attention, starting with the celebrated papers by Berestycki and Lions \cite{berestycki-lions83,berestycki-lions83b} on the case $N\geq 3$ and 
  by Berestycki, Gallou\"et and Kavian \cite{be_ga_ka83} for the case $N=2$. 
  We refer the reader to the monographs \cite{ambrosetti-malchiodi-b,kuzin-pohozaev,rabinowitz,struwe,willem} 
  and the references therein for a detailed account on the study of such equations.
  In contrast, much less is known about the case $\lambda<0$, due in particular to the fact that the usual
  variational method in $H^1(\R^N)$ breaks down, since the solutions of \eqref{eq:33a}, if any, will not
  decay faster than $O(|x|^{\frac{1-N}2})$ as $|x|\to\infty$ (see \cite{kato59}). 
  Recent results obtained by T. Weth and the author \cite{evequoz-weth-dual}
  confirmed nevertheless the existence for $\lambda<0$ of nontrivial $W^{2,p}(\R^N)$-solutions 
  for the problem \eqref{eq:33a} in $\R^N$ with $N\geq 3$. In a previous paper \cite{evequoz-weth14}, 
  existence results for \eqref{eq:33a} in all dimensions $N\geq 2$ and for more general nonlinearities
  were obtained by studying a Dirichlet-to-Neumann boundary-value problem,
  but only nonlinearities having compact support were considered.
  Let us also mention results concerning complex-valued solutions of \eqref{eq:33a} with prescribed
  asymptotic behavior, obtained using contraction mapping arguments, 
  by Guti\'errez \cite{gutierrez04} in dimension $N=3, 4$ and with $p=4$, 
  and by Jalade \cite{jalade04} in dimension $N=3$ for more general, compactly supported nonlinearities.

  Our present goal is to extend the results of \cite{evequoz-weth-dual} to the two-dimensional case and, 
  at the same time, to provide a basis for further study of the planar nonlinear Helmholtz equation.
  Without loss of generality, we shall focus on the case $\lambda=-1$ and therefore deal
  with the problem
  \begin{equation}\label{eq:33b}
  -\Delta u - u=Q|u|^{p-2}u\quad\text{ in }\R^2.
  \end{equation} 
  As in \cite{evequoz-weth-dual}, we shall reformulate \eqref{eq:33b} as an integral equation, 
  involving the resolvent operator $\cR$ associated 
  to the inhomogeneous Helmholtz equation
  $$
  -\Delta u -u =f\quad\text{in }\R^2
  $$
  and the outgoing radiation condition, which in two dimensions reads as:
  \begin{equation}\label{eqn:sommerfeld}
  \nabla u(x)\cdot\frac{x}{|x|}-iu(x)=o(|x|^{-\frac12}),\quad \text{as }|x|\to\infty
  \end{equation}
  (see \cite[Chap. 3.4]{colton-kress}). More precisely, we shall look for solutions of \eqref{eq:33b} 
  that satisfy the fixed-point equation
  \begin{equation}\label{eqn:nlh_integral}
  u=\mR\left(Q|u|^{p-2}u\right), \quad u\in L^p(\R^2).
  \end{equation}
  Here, $\mR$ denotes the real part of the operator $\cR$.
  For more details concerning the link between \eqref{eqn:nlh_integral} and \eqref{eq:33b}
  we refer the reader to the introduction of \cite{evequoz-weth-dual}. 
  
  Our first main result deals with the regularity and the asymptotic behavior of solutions 
  of the nonlinear problem \eqref{eqn:nlh_integral}.
  There, and in the sequel, $\cF$ denotes the Fourier transform on the 
  space of tempered distributions.
  \begin{theorem}\label{thm:far-field-nonlin}
  	Let $6\leq p<\infty$, $Q\in L^\infty(\R^2)$ and consider a solution $u$ of~\eqref{eqn:nlh_integral}.
  	Then, $u\in W^{2,q}(\R^2)$ for all $6\leq q<\infty$ and it is a strong solution of \eqref{eq:33b}. 
  	
  	Moreover, if $p>6$, we have $u\in W^{2,q}(\R^2)$ for all $4<q<\infty$ and
  	\begin{equation}\label{eqn:far-field-u-point}
  	u(x)= \sqrt{\frac{\pi}2}|x|^{-\frac12}\ \text{\em Re}\bigl[e^{i|x|+\frac{i\pi}{4}}
  	\mathfrak{f}_u({\textstyle\frac{x}{|x|}})\bigr]+o(|x|^{-\frac12}),\text{ as }|x|\to\infty,
  	\end{equation}
  	where $\mathfrak{f}_u(\xi)=\cF(Q|u|^{p-2}u)(\xi)$ for $\xi\in\R^2$ with $|\xi|=1$.
  \end{theorem}
  \begin{remark}
  	As for the Helmholtz equation in dimension $3$ (see \cite[p.~694]{evequoz-weth-dual}), 
  	the pointwise expansion \eqref{eqn:far-field-u-point} is satisfied for all
  	noncritical exponents $p\in(6,\infty)$. 
  	In the case $p=6$, it holds for radial solutions, 
  	under additional assumptions on the function $Q$. 
  	Indeed, assuming $Q$ to be $C^1$, radially symmetric and radially decreasing, 
  	we find by \cite[Theorem 4]{evequoz-weth14}, that every radial 
  	solution of \eqref{eqn:nlh_integral} satisfies $|u(x)|\leq C|x|^{-\frac12}$. 
  	From Proposition \ref{prop:farfield_N} below, we then obtain 
  	\eqref{eqn:far-field-u-point}.
  	In general, however, only the following weaker form of \eqref{eqn:far-field-u-point} holds
  	(cf. \cite[Lemma 4.3]{evequoz-weth-dual}):
  	\begin{equation*}
  	\lim_{R\to\infty}\frac{1}{R}\int\limits_{B_R(0)}\Bigl|u(x)
  	-\sqrt{\frac{\pi}2}|x|^{-\frac12}\ \text{Re}\bigl[e^{i|x|+\frac{i\pi}{4}}
  	\mathfrak{f}_u({\textstyle\frac{x}{|x|}})\bigr]\Bigr|^2\, dx=0.
  	\end{equation*}
  \end{remark}
  Our second main result concerns the existence of solutions for \eqref{eqn:nlh_integral}, 
  and hence for \eqref{eq:33b}, under two different assumptions on the nonnegative function $Q$. 
  \begin{theorem}\label{thm:exist}
  	For $6\leq p<\infty$ and $Q \in L^\infty(\R^2)\backslash\{0\}$, $Q\geq 0$, the following holds.
  	\begin{itemize}
  		\item[(a)] If $Q(x)\to 0$ as $|x|\to\infty$, the problem \eqref{eqn:nlh_integral}
  		admits a sequence of pairs of solutions $\pm u_n \in W^{2,q}(\R^2)$, for all $6\leq q<\infty$ if $p=6$ 
  		and all $4<q<\infty$ if $p>6$, such that $\|u_n\|_p \to \infty$ as $n \to \infty$.
  		
  		\item[(b)] If $Q$ is $\Z^2$-periodic and $p>6$, then \eqref{eqn:nlh_integral}
  		has a nontrivial solution pair $\pm u\in W^{2,q}(\R^2)$ for all $4<q<\infty$.
  	\end{itemize}
  \end{theorem}
  \begin{remark}
  	Theorem \ref{thm:far-field-nonlin} and \ref{thm:exist} can be extended to more general
  	nonlinearities, like those studied in \cite{e-helmholtz-orlicz-compact} for $N\geq 3$. 
  	Also, replacing the assumption $Q(x)\to 0$ as $|x|\to\infty$ by 
  	$Q\in L^{\frac32}(\R^2)\cap L^\infty(\R^2)$ in Theorem \ref{thm:exist}(a), 
  	one can prove that for every $p\in(2,\infty)$ the problem \eqref{eqn:nlh_integral} has infinitely 
  	many pairs of solutions $\{\pm u_n\}$ such that $Q^\frac1p u_n\in L^p(\R^2)$
  	for all $n$ and $\int_{\R^2}Q|u_n|^p\, dx\to \infty$ as $n\to\infty$.
  \end{remark}
  The proof of the above results is based on the method developed in the recent paper \cite{evequoz-weth-dual}, 
  but we emphasize that these results do not follow from their higher-dimensional counterparts.
  Indeed, the presence of a logarithmic singularity at $0$ in the kernel 
  of the resolvent operator in $\R^2$ (cf. \cite{colton-kress}) requires new estimates, 
  different from those obtained in \cite{evequoz-weth-dual},
  and which we believe to be also of independent interest.

  The paper is organized as follows.
  In the next section, we define the resolvent Helmholtz operator $\cR$ and derive
  $L^p$-estimates similar to \cite{gutierrez04,KRS87}.
  Next, the asymptotic expansion and the decay of solutions of linear equations are studied 
  and the section concludes with the proof of Theorem \ref{thm:far-field-nonlin}.
  Section \ref{sec:variational} is devoted to the existence proof for solutions of \eqref{eqn:nlh_integral}. 
  It starts with the extension of the dual variational method of \cite{evequoz-weth-dual} to $\R^2$ 
  and continues with the proof of Theorem \ref{thm:exist}(a), as an application of the symmetric Mountain Pass Theorem.
  There, an interaction estimate, more involved than in the case $N\geq 3$ is used to construct finite-dimensional
  subspaces of arbitrary dimension on which the quadratic part of the energy functional is positive.
  Next, the periodic case is studied and a nonvanishing property for the quadratic form associated to the resolvent $\cR$
  is derived. As in \cite{evequoz-weth-dual}, it constitutes a key ingredient in the proof of the existence
  of solutions in the periodic case, by which the paper concludes.

  \section{The planar resolvent and the far-field relation}\label{sec:resolvent}
  For a function $f$ in the Schwartz space $\cS(\R^2)$, the unique solution of the Helmholtz
  equation $-\Delta u -u=f$ in $\R^2$ which satisfies the radiation condition \eqref{eqn:sommerfeld} 
  is given by the convolution $u=\Phi\ast f$, where
  \begin{equation*}
  \Phi(x)=\frac{i}{4} H^{(1)}_0(|x|), \quad x\in\R^2.
  \end{equation*}
  Here, $H^{(1)}_0$ denotes the Hankel function of the first kind of order $0$ 
  (see e.g., \cite[Chap. 3.4]{colton-kress}).
  In view of the asymptotic behavior of $\Phi$ given by
  \begin{equation}\label{eq:3}
  \Phi(x) =\left \{
  \begin{aligned} 
  &\frac1{2\sqrt{2\pi}} |x|^{-\frac12}\, e^{i|x|+i\frac{\pi}{4}}[1+O(|x|^{-1})] 
  &&\qquad \text{as $|x| \to \infty$,}\\
  &\frac1{2\pi}\log\Bigl(\frac{2}{|x|}\Bigr)\left[1+O\Bigl(\frac1{\bigl|\log |x| \bigr|}\Bigr)\right]  
  &&\qquad \text{as $x \to 0$},
  \end{aligned}
  \right.
  \end{equation}
  (see, e.g., \cite[Eq. (5.16.3)]{lebedev}), there is a constant $C_0>0$ such that 
  \begin{equation}\label{eq:4}
  |\Phi(x)| \le C_0 \min \{1+\bigl|\log |x|\bigr|,|x|^{-\frac12}\} \qquad
  \text{for $x \in \R^2 \setminus \{0\}$.}   
  \end{equation}
  Following ideas of Kenig-Ruiz and Sogge \cite[Theorem 2.3]{KRS87} and Guti\'errez \cite[Theorem 6]{gutierrez04}, 
  we prove estimates which show that the resolvent Helmholtz operator $f\mapsto \Phi\ast f$, defined for $f\in\cS(\R^2)$,
  has for certain $1\leq t,q\leq\infty$ a continuous extension $\cR$: $L^t(\R^2)$ $\to$ $L^q(\R^2)$.
  In the following, given $1\leq p\leq\infty$, we let $p'$ denote its conjugate exponent.
  \begin{theorem}\label{thm:gut6}
  	Let $1\leq t<\frac43$ and $4<q\leq\infty$ satisfy $\frac23\leq \frac1t-\frac1q<1$. 
  	There is a constant $C=C(t,q)>0$ such that
  	\begin{equation}\label{eqn:resolv_pq}
  	\|\cR f\|_q\leq C \|f\|_t\quad\text{ for all }f\in\cS(\R^2).
  	\end{equation}
  	In particular, if $6\leq q<\infty$ then \eqref{eqn:resolv_pq} holds with $t=q'$.
  \end{theorem}
  In order to prove the above estimates, we consider a decomposition of the fundamental solution $\Phi$
  which we shall use also further below (see Theorem \ref{thm:nonvanishing}).
  
  \bigskip
  
  Fix $\psi \in \cS(\R^2)$ such that the Fourier transform $\wh{\psi}\in\cC^\infty_c(\R^2)$ 
  is radial, $0\leq \wh{\psi}\leq 1$, $\wh{\psi}(\xi)=1$ for $| |\xi|-1|\leq\frac16$ 
  and $\wh{\psi}(\xi)=0$ for $| |\xi|-1|\geq\frac14$. 
  Write $\Phi=\Phi_1+\Phi_2$, where
  \begin{equation}\label{eqn:decomp_phi}
  \Phi_1:= 2\pi (\psi * \Phi), \quad\text{and}\quad \Phi_2 = \Phi-\Phi_1.
  \end{equation}
  Since $\psi$ is a Schwartz function, we obtain from \eqref{eq:4}, 
  making $C_0$ larger if necessary,
  \begin{equation}\label{eqn:phi1}
  |\Phi_1(x)|\leq C_0 (1+|x|)^{-\frac12},\quad x\in\R^2.
  \end{equation}
  On the other hand, since $\cF(\Phi)(\xi)=\frac1{2\pi}(|\xi|^2-1-i0)^{-1}$ as a tempered
  distribution (see \cite{gelfand}) and since $\cF(\Phi_2)=(1-\wh{\psi})\cF(\Phi)$,
  it follows that $\cF(\Phi_2)\in C^\infty(\R^2)$ and $\cF(\Phi_2)(\xi)=(|\xi|^2-1)^{-1}$
  for $|\xi|\geq \frac54$. Consequently, $\partial^\gamma\cF(\Phi_2)\in L^1(\R^2)$
  for all $\gamma\in\N_0^N$ such that $|\gamma|\geq 1$ and this gives $|\Phi_2(x)|\leq \kappa_s|x|^{-s}$
  for all $s>0$, with some constant $\kappa_s>0$. Using also \eqref{eq:4} and \eqref{eqn:phi1}, we obtain,
  making again $C_0$ larger,
  \begin{equation}\label{eqn:phi2}
  |\Phi_2(x)|\leq C_0 \min\{1+\bigl| \log|x|\bigr|,|x|^{-3}\}, \quad x\in\R^2\backslash\{0\}.
  \end{equation}
  
  \bigskip
  
  \begin{proof}[Proof of Theorem~\ref{thm:gut6}]
  	The proof is inspired by Theorem 6 in \cite{gutierrez04}. In the sequel,
  	$C$ will denote a constant, whose value may change from line to line.
  	
  	Using \eqref{eqn:phi2} we see that $\Phi_2\in L^r(\R^2)$ for all $1\leq r<\infty$, 
  	and therefore Young's inequality gives for $1\leq t, q\leq \infty$ such that $0\leq \frac1t-\frac1q<1$,
  	\begin{equation}\label{eqn:estim_phi2}
  	\|\Phi_2\ast f\|_q\leq \|\Phi_2\|_r \|f\|_t\leq C \|f\|_t \quad\text{ for all }f\in\cS(\R^2).
  	\end{equation}
  	
  	To estimate the convolution with $\Phi_1$, let us fix a radial, nonnegative function 
  	$\eta \in C_c^\infty(\R^2)$ such that $\eta(x)= 1$ if $0 \le |x| \le 1$, 
  	$\eta(x)= 0$ if $|x|\ge 2$. For $j \in \N$, define $\phi_j \in C_c^\infty(\R^2)$
  	by $\phi_j(x)= \eta(x/2^j)- \eta(x/2^{j-1})$. Let also $\phi_0=\eta$.
  	We then have the dyadic decomposition 
  	\begin{equation*}
  	\Phi_1= \sum \limits_{j=0}^\infty \Phi_1^j \qquad \text{with 
  		$\Phi_1^j:= \Phi_1 \phi_j$ for $j \in \N\cup\{0\}$.} 
  	\end{equation*}
  	Choosing also $\varphi \in \cS(\R^2)$ 
  	such that its Fourier transform $\wh{\varphi} \in C^\infty_c(\R^2)$ 
  	is radial, nonnegative and satisfies
  	$\wh{\varphi}(\xi) = 1$ on $\{\xi \::\: ||\xi|-1|\le \frac12\}$ 
  	and $\wh{\varphi}(\xi) = 0$ on $\{\xi \::\: ||\xi|-1|\ge \frac34\}$, 
  	we see that $(\Phi_1\ast\varphi)\ast f=2\pi\Phi_1\ast f$ for all $f\in\cS(\R^2)$,
  	since $\supp \cF(\Phi_1)\subset\{\xi \::\: ||\xi|-1|\le \frac14\}$.
  	Hence, we look at the decomposition
  	\begin{equation}\label{eqn:dyadic_Q}
  	\Phi_1\ast \varphi= \sum \limits_{j=0}^\infty Q^j \qquad \text{with 
  		$Q^j:= \Phi_1^j\ast\varphi$ for $j \in \N\cup\{0\}$.} 
  	\end{equation}
  	From the decay properties of $\Phi_1$, we see that
  	\begin{equation}\label{eqn:l1_linf_phi1}
  	\|Q^j\|_\infty \leq \|\varphi\|_1\|\Phi_1^j\|_\infty\leq C 2^{-\frac{j}{2}} 
  	\qquad \text{for all $j\geq 1$},
  	\end{equation}
  	where $C$ is independent of $j$. 
  	On the other hand, Plancherel's identity and the Stein-Tomas Theorem \cite{tomas75} imply
  	for $1\leq t\leq\frac65$ and $f\in\cS(\R^2)$,
  	\begin{align}
  	\|Q^j\ast f\|_2^2
  	&=(2\pi)^2\int_{||\xi|-1|\leq\frac34}\left|\wh{\Phi}_1^j(\xi)\wh{g}(\xi)\right|^2\, d\xi
  	\leq C\int_{\frac14}^{\frac74} r^{\frac4t-3}\left|\wh{\Phi}_1^j(r)\right|^2\, \|g\|_t^2\, dr\nonumber\\
  	&\leq C\|\varphi\|_1^2 \|f\|_t^2\int_{\R^2}\left|\Phi_1^j(x)\right|^2\, dx 
  	\leq C 2^j\|f\|_t^2,\label{eqn:stein_tomas_phi1}
  	\end{align}
  	where we have set $g=\varphi\ast f$, and $C$ does not depend on $j$. 
  	From these two estimates and the Riesz-Thorin theorem 
  	\cite[Theorem V.1.3]{stein-weiss}, it follows that
  	\begin{equation*}
  	\|Q^j\ast f\|_q\leq C 2^{j\left(\frac2q-\frac12\right)}\|f\|_t\quad\text{ for all }f\in\cS(\R^2),
  	\end{equation*}
  	$1\leq t\leq\frac65$ and $2\leq q\leq \frac{t'}3$. 
  	Observe that the exponent is negative if $q>4$.
  	Since $\Phi_0^j\in L^1(\R^2)\cap L^\infty(\R^2)$, we conclude that for
  	$1\leq t< \frac{12}{11}$ and $4<q\leq \frac{t'}{3}$,
  	$$
  	\|\Phi_1\ast f\|_q=\frac1{2\pi}\|(\Phi_1\ast\varphi)\ast f\|_q
  	\leq \sum_{j=0}^\infty\|Q^j\ast f\|_q
  	\leq C\|f\|_t \sum_{j=0}^\infty 2^{j\left(\frac2q-\frac12\right)},
  	$$
  	and therefore
  	\begin{equation}\label{eqn:resolv_phi1}
  	\|\Phi_1\ast f\|_q\leq C \|f\|_t.
  	\end{equation}
  	By duality and convexity,
  	this estimate holds for all $1\leq t<\frac43$ and $4<q\leq\infty$ 
  	such that $\frac1t-\frac1q>\frac23$.
  	Taking into account the estimate \eqref{eqn:estim_phi2} for $\Phi_2$, we obtain
  	\eqref{eqn:resolv_pq}
  	for all $1\leq t<\frac43$, $4<q\leq\infty$ such that $\frac23<\frac1t-\frac1q<1$.
  	
  	To conclude the proof, it remains to show that \eqref{eqn:resolv_phi1} 
  	also holds when $\frac1t-\frac1q=\frac23$. We proceed similarly to \cite[Theorem 6]{gutierrez04}.
  	Using real interpolation (see \cite[Section V.3]{stein-weiss}), it is enough to prove
  	the restricted weak-type estimates
  	\begin{equation}\label{eqn:lorentz}
  	\|\Phi_1\ast f\|_{q,\infty}\leq C\|f\|_{t,1}\quad\text{ for all }f\in\cS(\R^2)
  	\end{equation}
  	for the endpoints $(t,q)=(\frac{12}{11},4)$ and $(\frac43,12)$, where $\|\cdot\|_{r,s}$ denotes
  	the norm on the Lorentz space $L^{r,s}(\R^2)$.
  	Moreover, by \cite[Theorem V.3.13 and Theorem V.3.21]{stein-weiss}, 
  	it suffices to prove \eqref{eqn:lorentz} 
  	for characteristic functions of measurable sets:
  	\begin{equation}\label{eqn:lorentz_char}
  	\lambda \left|\{x\in\R^2\, :\, |(\Phi_1\ast 1_E)(x)|>\lambda\}\right|^{\frac1q}
  	\leq C |E|^\frac1t\quad\text{ for all }\lambda>0,\, |E|<\infty.
  	\end{equation}
  	Here, $|E|$ denotes the measure of $E\subset\R^2$ and $1_E$ its characteristic
  	function.
  	Setting $A:=\{x\in\R^2\, :\, |(\Phi_1\ast 1_E)(x)|>\lambda\}$, choosing $\varphi\in\cS(\R^2)$ 
  	as above and recalling the dyadic decomposition \eqref{eqn:dyadic_Q}, we can write
  	\begin{align*}
  	|A|=\int_A\, dx\leq \frac1\lambda\int_{\R^2}\left|(\Phi_1\ast 1_E)(x)\right| 1_A(x)\, dx
  	\leq \frac1\lambda\sum_{j=0}^\infty \int_{\R^2}\left|(Q^j\ast 1_E)(x)\right| 1_A(x)\, dx.
  	\end{align*}
  	From the estimate \eqref{eqn:stein_tomas_phi1} with $t=\frac65$ and by duality, we obtain
  	\begin{equation}\label{eqn:estimQ_dual}
  	\|Q^j\ast f\|_2\leq C 2^{\frac{j}2}\|f\|_{\frac65} 
  	\quad\text{and}\quad \|Q^j\ast f\|_6\leq C 2^{\frac{j}2}\|f\|_2
  	\end{equation}
  	for all $f\in\cS(\R^2)$, the constant $C$ being independent of $j$.
  	Moreover, \eqref{eqn:l1_linf_phi1} gives
  	\begin{equation}\label{eqn:estimQ_1_inf}
  	\|Q^j\ast f\|_\infty\leq C 2^{-\frac{j}2}\|f\|_1, \quad j\geq 1.
  	\end{equation}
  	For $M\in\N_0$, we therefore obtain, using H\"older's inequalityand approximating $1_E$ by
  	Schwartz functions,
  	\begin{align*}
  	\sum_{j=0}^\infty\int_{\R^2}|(Q^j\ast 1_E)(x)|1_A(x)\, dx 
  	&\leq  \sum_{j=0}^M \|Q^j\ast 1_E\|_2 |A|^{\frac12} + \sum_{j=M+1}^\infty \|Q^j\ast 1_E\|_\infty |A|\\
  	&\leq C \left\{2^{\frac{M}2} |E|^{\frac56}|A|^{\frac12} +  2^{-\frac{M+1}2}|E| |A|\right\}.
  	\end{align*}
  	Choosing $M\in\N_0$ with $2^{\frac{M}2}\leq |E|^\frac1{12}|A|^\frac14\leq 2^{\frac{M+1}2}$,
  	the preceding estimates give
  	\begin{align*}
  	\lambda|A|\leq\sum_{j=0}^\infty\int_{\R^2}|(Q^j\ast 1_E)(x)|1_A(x)\, dx 
  	\leq C|E|^{\frac{11}{12}}|A|^\frac34,
  	\end{align*}
  	which yields $\lambda|A|^{\frac14}\leq C|E|^{\frac{11}{12}}$ and shows \eqref{eqn:lorentz_char} 
  	for the exponents $t=\frac{12}{11}$, $q=4$.
  	
  	Similarly, H\"older's inequality and the estimates \eqref{eqn:estimQ_dual}, 
  	\eqref{eqn:estimQ_1_inf} give
  	\begin{align*}
  	\sum_{j=0}^\infty\int_{\R^2}|(Q^j\ast 1_E)(x)|1_A(x)\, dx 
  	&\leq  \sum_{j=0}^M \|Q^j\ast 1_E\|_6 |A|^{\frac56} + \sum_{j=M+1}^\infty \|Q^j\ast 1_E\|_\infty |A|\\
  	&\leq C \left\{2^{\frac{M}2} |E|^{\frac12}|A|^{\frac56} +  2^{-\frac{M+1}2}|E| |A|\right\}.
  	\end{align*}
  	Choosing this time $M\in\N_0$ such that 
  	$2^{\frac{M}2}\leq |E|^\frac14|A|^\frac1{12}\leq 2^{\frac{M+1}2}$,
  	we find that
  	$\lambda|A|^{\frac1{12}}\leq C|E|^{\frac34}$, proving \eqref{eqn:lorentz_char} 
  	for the exponents $p=\frac43$, $q=12$. The proof is complete.
  \end{proof}
  We note that the conclusion of Theorem \ref{thm:gut6} is false for $(t,q)=(1,\infty)$. 
  Indeed, if $f\in C^\infty_c(\R^2)\backslash\{0\}$ is nonnegative and $f(x)=0$ for $|x|\geq 1$, 
  consider the sequence $(f_k)_k$, where $f_k(x)=k^2f(kx)$, $x\in\R^2$. 
  Then, $\|f_k\|_1=\|f\|_1$ for all $k$, but for every $x\neq 0$, we find
  $$
  (\Phi\ast f_k)(x)=\int_{B_1(0)}\Phi(x-k^{-1}y)f(y)\, dy\longrightarrow \Phi(x)\|f\|_1, 
  \quad \text{as }k\to\infty,
  $$
  by the Dominated Convergence Theorem. Hence, $\|\Phi\ast f_k\|_\infty\to \infty$ as $k\to\infty$.
  
  \bigskip
  
  We now turn to the pointwise asymptotic expansion of solutions of the Helmholtz equation
  and first look at the the linear problem $u=\cR f$. 
  \begin{proposition}\label{prop:farfield_N}
  	Let $f\in L^1(\R^2)$ satisfy $|f(x)|\leq \kappa|x|^{-2-\eps}$ for some $\kappa, \eps>0$. 
  	Then
  	$$
  	\cR f(x)= \sqrt{\frac{\pi}{2}}\ \frac{e^{i|x|+\frac{i\pi}{4}}}{|x|^{\frac12}}\ 
  	\wh{f}\bigl(\textstyle{\frac{x}{|x|}}\bigr) 
  	+ o(|x|^{-\frac12}), \quad\text{as }|x|\to\infty.
  	$$
  \end{proposition}
  \begin{proof}
  	Consider first for $x\in\R^2$ with $|x|\geq 2$,
  	$$
  	I_1(x)=\int\limits_{B_1(x)}\Phi(x-y)f(y)\, dy.
  	$$
  	From \eqref{eq:4}, we see that $|\Phi(z)|\leq C_0(1+|\log|z||)$ for all $|z|\leq 1$,
  	and we can write
  	\begin{align*}
  	|I_1(x)|&\leq C_0\kappa\int\limits_{B_1(x)}(1+|\log|x-y||)|y|^{-2-\eps} \, dy\\
  	&\leq C_0\kappa\left(\frac{|x|}{2}\right)^{-2-\eps}\int_{B_1(0)}(1+|\log|y||)\, dy,
  	\end{align*}
  	where the last integral is finite. In particular, $I_1(x)=o(|x|^{-\frac12})$ as $|x|\to\infty$.
  	Next, let  $A(x)=\{y\in\R^2\, :\, |x-y|>1\text{ and }|y|\geq\sqrt{|x|}\}$ and consider
  	$$
  	I_2(x)=\int\limits_{A(x)}\Phi(x-y)f(y)\, dy.
  	$$
  	The estimate \eqref{eq:4} implies $|\Phi(z)|\leq C_0|z|^{-\frac12}$ for all $|z|>1$, and therefore
  	\begin{align*}
  	|I_2(x)|&\leq C_0\int\limits_{A(x)}|x-y|^{-\frac12}|f(y)|\ dy\\
  	&\leq C_0|x|^{-\frac12}\int\limits_{A(x)}\left(1+|x-y|^{-\frac12}|y|^{\frac12}\right)|f(y)|\, dy\\
  	&\leq C_0|x|^{-\frac12}\Bigl(\int\limits_{A(x)}|f(y)|\, dy+\kappa 
  	\int\limits_{A(x)}|x-y|^{-\frac12}|y|^{-\frac32-\eps}\, dy\Bigr).
  	\end{align*}
  	Since $f\in L^1(\R^2)$, the first integral on the last line goes to zero uniformly 
  	as $|x|\to\infty$. The same is true for the second integral, since 
  	$A(x)\subset\R^2\backslash B_1(0)$ and since $-\frac12+(-\frac32-\eps)<-2$ 
  	(see, e.g., \cite[Appendix 2, Lemma 1]{alsholm-schmidt70}). 
  	Hence, $I_2(x)=o(|x|^{-\frac12})$ as $|x|\to\infty$.
  	Concerning the remaining integral
  	$$
  	I_3(x)=\int\limits_{D(x)}\Phi(x-y)f(y)\, dy,
  	$$
  	where $D(x)=\{y\in\R^2\, :\, |x-y|>1\text{ and }|y|\leq\sqrt{|x|}\}$, 
  	we can write using \eqref{eq:3}, 
  	\begin{equation}\label{eqn:I3}
  	I_3(x)=\frac{e^{\frac{i\pi}{4}}}{2\sqrt{2\pi}}
  	\int\limits_{D(x)}\frac{e^{i|x-y|}}{|x-y|^{\frac12}}\Bigl(1+\delta(|x-y|)\Bigr)f(y)\,dy,
  	\end{equation}
  	where $\sup\limits_{r\geq 1}r|\delta(r)|<\infty$.
  	Furthermore, setting $\wh{x}:=\frac{x}{|x|}$ for $x\neq 0$, one finds
  	\begin{equation*} 
  	\Bigl|\ |x-y|-|x|+\wh{x}\cdot y\ \Bigr|\leq |x|^{-1}|y|^2\quad\text{ 
  	for all }x, y\in\R^2\text{ with }x\neq 0\text{ and }|y|\leq\frac{|x|}2. 
  	\end{equation*}
  	Arguing as in \cite[Proposition 2.8]{evequoz-weth-dual} and using the estimate 
  	$|f(y)|\leq\kappa|y|^{-2-\eps}$, we obtain
  	$$
  	\left|I_3(x)-\frac{1}{2\sqrt{2\pi}}\frac{e^{i|x|+i\frac{\pi}{4}}}{|x|^{\frac12}}
  	\int_{D(x)}e^{-i\wh{x}\cdot y}f(y)\, dy\right|
	\leq \tilde{\kappa} |x|^{-\frac12-\frac{\eps}{4}},
  	$$
  	for some constant $\tilde{\kappa}>0$. 
  	Putting together the estimates for $I_1$, $I_2$ and $I_3$ and using the integrability of $f$,
  	we deduce that
  	$$
  	\cR f(x)=I_1(x)+I_2(x)+I_3(x)
  	=\sqrt{\frac{\pi}{2}}\frac{e^{i|x|+i\frac{\pi}{4}}}{|x|^{\frac12}}\wh{f}(\wh{x}) 
  	+o(|x|^{-\frac12}), \quad\text{as }|x|\to\infty,
  	$$
  	and this concludes the proof.
  \end{proof}
  
  The next result gives an upper bound for the decay of solutions of convolution equations involving
  a kernel with the asymptotic properties of $\Phi$. Combined with the regularity result below, it will
  provide a decay bound for solutions of the nonlinear problem \eqref{eqn:nlh_integral}.
  \begin{lemma}\label{lem:asympt_fct}
  	Let $u, V$: $\R^2$ $\to$ $\R$ be measurable functions 
  	satisfying $V\in L^q(\R^2)$, $Vu\in L^{\tilde q}(\R^2)$, where $1<q, \tilde q<\frac43$. 
  	If $u=K\ast(Vu)$ and 
  	$$
  	|K(x)|\leq C_0\min\{1+\bigl| \log|x| \bigr|, |x|^{-\frac12}\}\quad\text{for } x\neq 0,
  	$$
  	then there exists a constant $C>0$ such that 
  	$|u(x)|\leq C|x|^{-\frac12}$ for all $x\neq 0$.
  \end{lemma}
  \begin{proof}
  	Let $B_R:=B_R(0)$ and $M_R:=\R^2\backslash B_R$ for $R>0$, and define $\wt{K}(x)=C_0\min\{1+\bigl| \log|x| \bigr|, |x|^{-\frac12}\}$ for $x\neq 0$. H\"older's inequality, then gives
  	\begin{align*}
  	&\int_{M_R}\wt{K}(x-y)|V(y)|\, dy
  	\leq C_0\int_{M_R}|V(y)|\min\{1+ \bigl|\log |x-y| \bigr|, |x-y|^{-\frac12} \}\, dy\\
  	&\leq C_0\left(\int_{M_R}|V(y)|^q\right)^\frac1q\Bigl(\int_{B_1(0)}(1+ \bigl|\log |y|\bigr|)^{q'}\, dy
  	+\int_{\R^2\backslash B_1(0)}|y|^{-\frac{q'}{2}}\, dy\Bigr)^{\frac{1}{q'}},
  	\end{align*}
  	which, as $R\to\infty$, tends to $0$ uniformly in $x$, since $4<q'<\infty$.
  	Hence, we may fix $R>1$ such that
  	\begin{equation}\label{eqn:VR}
  	\sup_{x\in\R^2}\int_{M_R}\wt{K}(x-y)|V(y)|\, dy<\frac14.
  	\end{equation}
  	The decay estimate on $u$ will follow with the help of an iteration procedure
  	similar to the one of Zemach and Odeh \cite{zemach-odeh60}. 
  	For $|x|\geq R$ we set
  	$$
  	u_0(x)=\int_{B_R}K(x-y)V(y)u(y)\, dy, \qquad B_0(x)=\int_{M_R}K(x-y)V(y)u(y)\, dy,
  	$$
  	and define inductively for $k\geq 1$,
  	$$
  	u_k(x)=\int_{M_R}K(x-y)V(y)u_{k-1}(y)\, dx,
  	\text{ } B_k(x)=\int_{M_R}K(x-y)V(y)B_{k-1}(y)\, dx.
  	$$
  	Thus, for each $m\in\N$,
  	$$
  	u=\sum_{k=0}^m u_k + B_m.
  	$$
  	Since $Vu\in L^{\tilde q}(\R^2)$ with $1<\tilde q<\frac43$, and $|K(x)|\leq\wt{K}(x)$ for all $x$, 
  	we find that $\beta_0:=\sup\limits_{|x|>R}|B_k(x)|<\infty$. 
  	Moreover, setting $\beta_k:=\sup\limits_{|x|>R}|B_k(x)|$, \eqref{eqn:VR} yields 
  	$\beta_k\leq \frac14\beta_{k-1}$, for all $k\geq 1$,
  	and therefore $\beta_k\to 0$, as $k\to\infty$. This gives
  	$$
  	u=\sum\limits_{k=0}^\infty u_k,\quad\text{ uniformly in }M_R.
  	$$
  	Moreover, since $u_0\in L^\infty(\R^2)$ and $Vu\in L^1(B_R)$, we have $\mu_0:=\sup\limits_{|x|\geq R}|x|^\frac12 |u_0(x)|<\infty$.
  	Setting, for $k\geq 1$, $\mu_k:=\sup\limits_{|x|\geq R}|x|^\frac12 |u_k(x)|$, and noticing that
  	$\wt{K}(z)=C_0|z|^{-\frac12}$ for all $|z|\geq 1$, we obtain
  	\begin{align*}
  	|x|^{\frac12}|u_k(x)|
  	&\leq \mu_{k-1} |x|^{\frac12}\int_{M_R}\wt{K}(x-y)|V(y)|\ |y|^{-\frac12}\, dy \\
  	&\leq \mu_{k-1} |x|^{\frac12}\Bigl[C_0\Bigl(\frac{|x|}2\Bigr)^{-\frac12}\int_{R\leq |y|\leq\frac{|x|}2}|V(y)|\ |y|^{-\frac12}\, dy \\
  	&+\Bigl(\frac{|x|}2\Bigr)^{-\frac12}\int_{|y|\geq\max\{R,\frac{|x|}2\}}\wt{K}(x-y)|V(y)|\, dy\Bigr] \\
  	& \leq 2\sqrt{2}\ \mu_{k-1} \sup_{z\in\R^2}\int_{M_R}\wt{K}(z-y)|V(y)|\, dy,
  	\end{align*}
  	for all $|x|\geq R$, and from \eqref{eqn:VR}, we deduce that $\mu_k\leq \frac{\mu_{k-1}}{\sqrt{2}}$. 
  	As a consequence,
  	$$
  	\sup_{|x|\geq R}|x|^{\frac12}|u(x)|\leq \sum_{k=0}^\infty \mu_k \leq 
  	\mu_0 \sum_{k=0}^\infty 2^{-\frac{k}{2}}<\infty,
  	$$
  	and this concludes the proof.
  \end{proof}
  
  The last preliminary step towards the proof of Theorem~\ref{thm:far-field-nonlin} consists in 
  establishing regularity properties for the solutions of \eqref{eqn:nlh_integral}. This will also
  prove the first part of Theorem~\ref{thm:far-field-nonlin}.
  \begin{lemma}\label{lem:regularity}
  	Let $6\leq p<\infty$, $Q\in L^\infty(\R^2)$ and consider a solution 
  	$u\in L^p(\R^2)$ of \eqref{eqn:nlh_integral}. Then $u$ is a strong solution of \eqref{eq:33b}
  	and it belongs to $W^{2,q}(\R^2)$ for all $6\leq q<\infty$, if $p=6$, and all $4<q<\infty$, if $p>6$.
  \end{lemma}
  \begin{proof}
  	Since $Q\in L^\infty(\R^2)$ and $6\leq p<\infty$, we find that $f:=Q|u|^{p-2}u$ belongs
  	to $L^{p'}(\R^2)$ and that $1<p'\leq\frac65$. Hence, Theorem \ref{thm:gut6} gives 
  	$u\in L^6(\R^2)\cap L^\infty(\R^2)$, and since $p-1\geq 5$, this means that
  	$f$ belongs to $L^{\frac65}(\R^2)\cap L^\infty(\R^2)$. Thus, we obtain from elliptic estimates
  	(see \cite[Chapter 9]{gilbarg-trudinger})
  	that $u\in W^{2,\sigma}(\R^2)$ for all $6\leq\sigma<\infty$. 
  	This proves the lemma in the case $p=6$.
  	
  	Assuming next $p>6$, we claim that $u\in L^q(\R^2)$ for all $4<q\leq\infty$. 
  	Supposing for the moment that it is true, we find $f=Q|u|^{p-2}u\in L^t(\R^2)$ 
  	for all $\frac{4}{p-1}<t\leq\infty$, where $p-1>1$. 
  	Hence, we get from elliptic estimates that $u\in W^{2,q}(\R^2)$ for all $4<q<\infty$ 
  	and the lemma is proved.
  	
  	We now prove the claim. As a consequence of the first step, 
  	we find that $f\in L^{p'}(\R^2)\cap L^\infty(\R^2)$. 
  	Applying Theorem~\ref{thm:gut6} with $t=p'<\frac65$,
  	and using the fact that $u\in L^\infty(\R^2)$,
  	we obtain that $u\in L^q(\R^2)$ for all $4<q\leq \infty$ 
  	such that $\frac1q\leq \frac1{p'}-\frac23$.
  	If $\frac1{p'}-\frac23\geq \frac14$, the claim follows. 
  	Otherwise, we iterate the argument and find that
  	$f\in L^t(\R^2)$ for all $1\leq t\leq\infty$ such that $\frac1t\leq(p-1)(\frac1{p'}-\frac23)$. 
  	Let $\frac1{r_1}=(p-1)(\frac1{p'}-\frac23)$ and remark that $r_1<p'$, since $p'<\frac65$. 
  	Applying Theorem~\ref{thm:gut6} with $t=r_1$, we obtain $u\in L^q(\R^2)$ for all $4<q\leq \infty$ 
  	such that $\frac1q\leq \frac1{r_1}-\frac23$. Iterating the procedure we find at each step
  	$u\in L^q(\R^2)$ for all $4<q\leq\infty$ such that $\frac1q\leq\frac1{r_m}-\frac23$, 
  	where $r_m$ is given by 
  	$$
  	r_0=p', \quad \frac1{r_m}=(p-1)\left(\frac1{r_{m-1}}-\frac23\right), m\geq 1.
  	$$
  	It satisfies $\frac1{r_m}-\frac1{r_{m-1}}=2(p-1)^m\left(\frac16-\frac1{p}\right)>0$ 
  	and therefore $\frac1{r_m}\to\infty$ as $m\to\infty$.
  	Thus, after finitely many iterations, we obtain $\frac1{r_m}-\frac23\geq \frac14$ 
  	and the claim follows.
  \end{proof}
  
 With the help of the above results, we can now give the proof of our first main theorem.
  \begin{proof}[Proof of Theorem~\ref{thm:far-field-nonlin}]
  	As remarked above, the regularity was already proved in Lemma~\ref{lem:regularity}. 
  	Restricting to the case where $p>6$, we see that the functions $Q|u|^{p-2}$ and $Q|u|^{p-2}u$ 
  	belong to $L^{\frac65}(\R^2)$, since $u\in L^q(\R^2)$ for all $q>4$ by Lemma~\ref{lem:regularity}.
  	Thus, Lemma~\ref{lem:asympt_fct} with $q=\tilde{q}=\frac65$, $K=\text{Re}(\Phi)$ and $V=Q|u|^{p-2}$
  	ensures that $u(x)=O(|x|^{-\frac12})$ as $|x|\to\infty$. Therefore, $f(x)=O(|x|^{-\frac{p-1}{2}})$ 
  	as $|x|\to\infty$. Since $\frac{p-1}{2}>2$, the expansion \eqref{eqn:far-field-u-point} follows 
  	from Proposition~\ref{prop:farfield_N}, after taking real parts.
  \end{proof}
  
  \bigskip
  \section{Variational setting in the plane and existence of solutions for the nonlinear problem}\label{sec:variational}
  In order to prove Theorem~\ref{thm:exist}, we extend the dual variational method developed 
  in \cite{evequoz-weth-dual} to the space dimension $2$.
  Let therefore $Q\in L^\infty(\R^2)\backslash\{0\}$ be a nonnegative function and consider 
  for $6\leq p<\infty$ the energy functional
  \begin{equation*}
  J\, :\ L^{p'}(\R^2)\ \to\ \R,\quad J(v)
  =\frac{1}{p'}\int_{\R^2}|v|^{p'}\, dx 
  - \frac12 \int_{\R^2}Q(x)^\frac1pv(x) \mR(Q^\frac1p v)(x)\, dx,
  \end{equation*}
  In this section, we consider, unless explicitly stated, real-valued functions and denote 
  by $\mR$ the real part of the resolvent operator $\cR$.
  Obviously, $J(-v)=J(v)$ for all $v\in L^{p'}(\R^2)$. Moreover, one can show that $J$ 
  is of class $C^1$ and that every critical point of $J$ corresponds to a solution 
  of \eqref{eqn:nlh_integral} in the following way.
  
  A function $v\in L^{p'}(\R^2)$ satisfies $J'(v)=0$ if and only if it solves the integral equation
  $|v|^{p'-2}v=Q^\frac1p \mR(Q^\frac1p v)$. Setting 
  $$
  u=\mR(Q^\frac1p v)\in L^p(\R^2),
  $$ 
  it follows that $u= \mR(Q|u|^{p-2}u)$, i.e., $u$ solves \eqref{eqn:nlh_integral}. 
  Note that, by Lemma~\ref{lem:regularity}, $u$ is then a strong solution of \eqref{eq:33b}.
  
  As a consequence of Theorem \ref{thm:gut6}, 
  the Birman-Schwinger type operator $\mK$: $L^{p'}(\R^2)$ $\to$ $L^p(\R^2)$ given by
  $$
  \mK v:=Q^\frac1p \mR(Q^\frac1p v), \quad v\in L^{p'}(\R^2),
  $$
  and appearing in the quadratic part of $J$, is continuous for $6\leq p<\infty$ and
  has compactness properties which will be important in the sequel.
  More precisely, the operator $1_B\mK$: $L^{p'}(\R^2)$ $\to$ $L^p(\R^2)$ is compact
  for every bounded and measurable set $B\subset\R^2$, and if,
  in addition, $Q(x)\to 0$ as $|x|\to\infty$, then $\mK$ itself is compact.
  Here, $1_B$ denotes the characteristic function of the set $B\subset\R^2$.
  
  To see this, consider a sequence $(v_n)_n$ converging weakly to $0$ in $L^{p'}(\R^2)$ and
  choose $R>0$ such that $B\subset B_R:=B_R(0)$. Due to the continuity of the resolvent,
  there holds $\mR(Q^{\frac1p}v_n)\weakto 0$ in $L^p(\R^2)$, and elliptic estimates ensure the
  boundedness of $\mR(Q^{\frac1p}v_n)$ in the Sobolev space $W^{2,p'}(B_R)$. Since $p<\infty$, the
  embedding $W^{2,p'}(B_R)\embed L^p(B_R)$ is compact, and since $Q$ is bounded and $B\subset B_R$, 
  we conclude that $1_B\mK v_n=1_BQ^{\frac1p}\mR(Q^{\frac1p}v_n)\to 0$, strongly in $L^p(\R^2)$,
  and obtain the first compactness property.
  In the case where $Q(x)\to 0$ as $|x|\to\infty$, the continuity of $\mR$ and the estimate
  $$
  \|(1-1_{B_R})\mK v_n\|_p\leq  \operatorname*{ess\, sup}\limits_{|x|\geq R}Q^{\frac1p}(x)\ \|\mR(Q^{\frac1p}v_n)\|_p,
  $$
  which holds for every $R>0$, ensure the convergence $\|(1-1_{B_R})\mK v_n\|_p\to 0$ as $R\to\infty$,
  uniformly in $n$. On the other hand, as we have already seen, $\|1_{B_R}\mK v_n\|_p\to 0$ as $n\to\infty$,
  for every $R>0$. Combining these two facts yields the strong convergence $\mK v_n\to 0$ in $L^p(\R^2)$
  and hence the compactness of $\mK$.

  We start by proving the existence of an unbounded sequence of solutions  in the case where
  $Q(x)\to 0$ as $|x|\to\infty$.
  \begin{proof}[Proof of Theorem~\ref{thm:exist}(a)]
  	The result will follow from the symmetric Mountain Pass Theorem (see \cite{ambrosetti-rabinowitz73} 
  	for the original work, and \cite{ghoussoub} for the version we use here) applied to 
  	the even functional $J$.
  	
  	For this, we first show that $0$ is a strict local minimum or, 
  	more precisely, that $J(v)\geq \delta>0$ for all $\|v\|=\rho$, provided $\rho>0$ is small.
  	Indeed, the operator $\mK$ being continuous, there exists a constant $C>0$ such that 
  	$\|\mK v\|_p\leq C\|v\|_{p'}$ for all $v\in L^{p'}(\R^2)$.
  	Hence, if $\|v\|_{p'}=\rho>0$, we obtain
  	$$
  	J(v)= \frac{1}{p'}\rho^{p'}-\frac12\int_{\R^2}v\mK v\, dx 
  	\geq \frac{1}{p'}\rho^{p'}-\frac{C}{2}\rho^2>0
  	$$
  	for all $\rho>0$ small enough, since $p'<2$.

  	In the next step, we prove for each integer $m$ the existence of 
  	an $m$-dimensional subspace $\cW_m$ of $L^{p'}(\R^2)$ and of a radius $R_m>0$ 
  	with the property that $J(v)\leq 0$ for all $v\in\cW_m$ such that 
  	$\|v\|_{p'}\geq R_m$. 
  	
  	Let $m$ be any integer. Since $Q \in L^\infty(\R^2)\backslash \{0\}$ is nonnegative, 
  	there is a point $x_0\in\R^2$ of metric density $1$ for the set $\{Q>0\}$. 
  	Hence, there is $0<\delta<1$ such that
  	\begin{equation}\label{eq:21}
  	|B_\delta(x_0)\cap\{Q>0\}| \geq  \frac12 |B_\delta(x_0)|.
  	\end{equation} 
  	Since $\text{Re}(\Phi)$ is bounded outside of every neighborhood of 
  	zero, by \eqref{eq:4}, and since 
  	$\text{Re}(\Phi(x))\approx \frac1{2\pi}\log(\frac2{|x|})\to +\infty$ 
  	as $|x| \to 0$, by \eqref{eq:3}, 
  	we may also assume that $\delta>0$ satisfies for 
  	$\Psi^*(\tau):=\inf\limits_{B_{\tau}(0)\setminus \{0\}}\text{Re}(\Phi)$ and
  	$\Psi_*(\tau):=\|\text{Re}(\Phi)\|_{L^\infty(\R^2 \backslash B_{\tau}(0))}$ the property,
  	\begin{equation}\label{eq:19}
  	\Psi^*(\sigma^m)>(m-1)\Psi_*(\sigma) \qquad \text{for $\sigma \in (0,\delta]$.}
  	\end{equation}
  	Let $\sigma:=\frac{\delta}{4\sqrt{m}}$, $\tau:=\frac12\sigma^m$ and
  	choose $m$ disjoint open balls $B^1,\dots,B^m \subset B_\delta(x_0)$ 
  	as follows.
  	By \eqref{eq:21}, we can choose $x_1\in B_\delta(x_0)\cap\{Q>0\}$ and $\tau_1\in(0,\tau]$ 
  	such that $B^1:=B_{\tau_1}(x_1)\subset B_{\delta}(x_0)$ and $|B^1\cap\{Q>0\}|>0$.
  	Let now $\omega_1:=(B_\delta(x_0)\cap\{Q>0\})\backslash B_{2\tau_1+\sigma}(x_1)$ and
  	observe that
  	$$
  	|\omega_1|\geq \frac12 |B_\delta(x_0)|-|B_{2\sigma}(x_1)|
  	\geq \left(\frac12-\frac1{4m}\right)\pi\delta^2>0.
  	$$
  	Thus we may choose $x_2\in\omega_1$ and $\tau_2\in(0,\tau]$ such that 
  	$B^2:=B_{\tau_2}(x_2)\subset B_{\delta}(x_0)$ and $|B^2\cap\{Q>0\}|>0$.
  	Inductively, we let for $2\leq k\leq m-1$, 
  	$\omega_k:=(B_\delta(x_0)\cap\{Q>0\})\backslash 
  	\bigcup\limits_{i=1}^{k}B_{2\tau_i+\sigma}(x_i)$ and remark that
  	$$
  	|\omega_k|\geq \frac12|B_\delta(x_0)|-\sum_{i=1}^k|B_{2\sigma}(x_i)|
  	\geq \left(\frac12-\frac{k}{4m}\right)\pi\delta^2>0.
  	$$
  	Therefore, we may choose $x_{k+1}\in\omega_k$ and $\tau_{k+1}\in(0,\tau]$
  	such that $B^{k+1}:=B_{\tau_{k+1}}(x_{k+1})\subset B_{\delta}(x_0)$ and 
  	$|B^{k+1}\cap\{Q>0\}|>0$. Notice that, by construction,
  	\begin{equation}\label{eq:19bis}
  	\text{diam}B^i\leq \sigma^m\;\text{ and }\;
  	\dist(B^i,B^j):= \inf\{|x-y|\::\:x \in B^i, \:y \in B^j\} \geq \sigma
  	\end{equation}
  	for all $i\neq j$. 
  	Let us now fix $z_1, \ldots, z_m \in C_c^\infty(\R^2)$
  	such that $z_i>0$ in $B^i$ and $z_i = 0$ in $\R^2 \backslash B^i$.
  	We define $\cW_m$ as the subspace spanned by $z_1, \ldots, z_m$.
  	Writing any $v \in \cW_m \backslash \{0\}$ as $v= \sum \limits_{i=1}^m a_i z_i$ 
  	with $a= (a_1,\dots,a_m) \in \R^m \backslash \{0\}$,
  	it follows from \eqref{eq:21}, \eqref{eq:19} and \eqref{eq:19bis} that
  	\begin{align*}
  	\int_{\R^2} v \mK v\, dx&=\sum_{i,j=1}^m a_i a_j \int_{B^i}\int_{B^j} 
  	\text{Re}(\Phi)(x-y)Q(x)^\frac1p Q(y)^\frac1p z_i(x)z_j(y)\,dxdy \\
  	&\geq \Psi^*(\sigma^m)  \sum_{i=1}^m a_i^2 \Bigl(\int_{B^i} Q(x)^\frac1p
  	z_i(x)\,dx\Bigr)^2\\
  	&- \Psi_*(\sigma) \sum_{\stackrel{i,j=1}{i \not= j}}^m |a_i| |a_j| 
  	\Bigl(\int_{B^i} Q(x)^\frac1p z_i(x)\,dx\Bigr) 
  	\Bigl(\int_{B^j} Q(x)^\frac1p z_j(x)\,dx\Bigr)\\
  	&\geq \sum_{i=1}^m \Bigl(\Psi^*(\sigma^m) -(m-1)\Psi_*(\sigma) \Bigr)a_i^2
  	\Bigl(\int_{B^i} Q(x)^\frac1p z_i(x)\,dx\Bigr)^2 >0
  	\end{align*}
  	(cf. \cite[Lemma 5.1]{evequoz-weth-dual}). 
  	From the continuity of $\mK$, we obtain
  	$$
  	\mu_m:= \inf_{v \in \cW_m, \|v\|_{p'}=1}\: \int_{\R^2}v \mK v\, dx >0,
  	$$
  	and therefore,
  	$$
  	J(v)= \frac{\|v\|_{p'}^{p'}}{p'}-\frac12\int_{\R^2}v\mK v\, dx 
  	\leq  \|v\|_{p'}^{p'}\Bigl(\frac{1}{p'}- \frac12\|v\|_{p'}^{2-p'}
  	\mu_m \Bigr)\qquad \text{for } v \in \cW_m.
  	$$
  	Since $p'<2$, the assertion follows by choosing $R_m>0$ large enough.

  	The last step to apply the Mountain Pass Theorem is to check the Palais-Smale condition for $J$.
  	Let $(v_n)_n\subset L^{p'}(\R^2)$ be a Palais-Smale sequence for $J$, 
  	i.e., $(J(v_n))_n$ is bounded and $J'(v_n)\to 0$ in 
  	$L^{p'}(\R^2)^\ast$ as $n\to\infty$, we need to show that $(v_n)_n$ has a
  	subsequence converging strongly in $L^{p'}(\R^2)$. Here, the assumption $Q(x)\to 0$,
  	$|x|\to\infty$, will come into play.
  	
  	First observe that since
  	\begin{equation}\label{eqn:bnd_PS}
  	J(v_n)=\left(\frac{1}{p'}-\frac12\right)\|v_n\|_{p'}^{p'}+\frac12J'(v_n)v_n
  	\geq\left(\frac{1}{p'}-\frac12\right)\|v_n\|_{p'}^{p'}-\frac12\|J'(v_n)\|_\ast\|v_n\|_{p'}
  	\end{equation}
  	and $1<p'<2$, the sequence $(v_n)_n$ is bounded in $L^{p'}(\R^2)$.
  	Hence, there is $v\in L^{p'}(\R^2)$ such that, up to a subsequence,
  	$v_n\weakto v$ in $L^{p'}(\R^2)$ and $\|v\|_{p'}\leq\liminf\limits_{n\to\infty}\|v_n\|_{p'}$.
  	From the convexity of the function $t\mapsto |t|^{p'}$ 
  	and the compactness of the operator $\mK$, guaranteed by the assumption 
  	that $Q(x)\to 0$ as $|x|\to\infty$, it follows that $\lim\limits_{n\to\infty}\|v_n\|_{p'}=\|v\|_{p'}$,
  	and therefore $v_n\to v$ strongly in $L^{p'}(\R^2)$.
  	Hence $J$ satisfies the Palais-Smale condition.
  	
  	Since the hypotheses of the symmetric Mountain Pass Theorem are fulfilled,
  	there exists a sequence of pairs $\{\pm v_n\}$ of nontrivial critical points of $J$ 
  	with $J(v_n) \to \infty$ and thus,  $\|v_n\|_{p'} \to \infty$ as $n \to \infty$. 
  	Setting $u_n:= \mR(Q^{\frac{1}{p}} v_n)$ and using Lemma \ref{lem:regularity} concludes the proof.
  \end{proof}
  
  In the case where $Q$ is periodic, the Palais-Smale condition is in general not satisfied.
  A crucial ingredient in the proof of Theorem~\ref{thm:exist}(b) will be the following 
  nonvanishing property of the resolvent, analogue to \cite[Theorem 3.1]{evequoz-weth-dual} and reminiscent
  of the Lions compactness lemma \cite[II. Lemma I.1]{lions1}.
  It will allow us to obtain the existence, up to translations, 
  of a nontrivial weak limit for a Palais-Smale sequence of $J$. 
  We prove the result in a more general form than needed, since we believe that it will
  be useful for the study of complex-valued solutions of \eqref{eq:33b} also.
  \begin{theorem}\label{thm:nonvanishing}
  	Let $6<p<\infty$ and consider a bounded sequence $(v_n)_n\subset L^{p'}(\R^2)$ satisfying 
  	$$ \limsup \limits_{n\to\infty}\left|\,\int_{\R^2}v_n\cR v_n\, dx\right|>0.$$
  	Then there exists $R>0$, $\zeta>0$ and
  	$(x_n)_n\subset\R^2$ such that, up to a subsequence, 
  	\begin{equation}\label{eqn:liminf}
  	\int_{B_R(x_n)}|v_n|^{p'}\, dx \geq \zeta\quad\text{for all }n.
  	\end{equation}
  \end{theorem}
  \begin{proof}
  	Recall the decomposition of the fundamental solution 
  	$\Phi=\Phi_1+\Phi_2$ introduced in \eqref{eqn:decomp_phi},
  	and the estimates \eqref{eqn:phi1} and \eqref{eqn:phi2}:
  	$$
  	|\Phi_1(x)|\leq C_0(1+|x|)^{-\frac12}\quad\text{and}\quad
  	|\Phi_2(x)|\leq C_0\min\left\{1+| \log|x| |, |x|^{-3}\right\}.
  	$$
  	The proof of the theorem consists in three claims. We first prove a variant of the conclusion
  	with $\Phi_2$ in place of $\Phi$ and for Schwartz functions. 
  	Next, a decay estimate for the convolution with $\Phi_1$ outside larger and larger balls
  	is established. It is used in the third step to obtain the conclusion of the theorem for $\Phi_1$
  	in place of $\Phi$, again for Schwartz functions.
  	
  	{\bf Claim 1}: Let $2< p < \infty$, and $(v_n)_n\subset \cS(\R^2)$ 
  	bounded in $L^{p'}(\R^2)$ satisfy
  	\begin{equation}\label{eqn:vanish_phi2}
  	\sup_{y\in\R^2}\int_{B_\rho(y)}|v_n|^{p'}\,dx\to 0 , \quad \text{as }n\to\infty, \quad\text{for all }\rho>0.
  	\end{equation}
  	Then $\int_{\R^2} v_n [\Phi_2 * v_n]\,dx \to 0$ as $n \to \infty$.
  	
  	For the proof, let $A_R:= \{x \in\R^2\::\: \frac{1}{R} \le |x| \le R\}$ and 
  	$D_R:= \R^2 \backslash A_R$ for $R >1$, and first remark that \eqref{eqn:phi2}
  	gives $\Phi_2\in L^t(\R^2)$ for all $1\leq t<\infty$. Hence, Young's inequality implies
  	\begin{equation}\label{eq:14}
  	\sup_{n \in \N} \Bigl|\int_{\R^2} v_n [(1_{D_R} \Phi_2)*v_n]\,dx\Bigr| 
  	\le\|\Phi_2\|_{L^{\frac{p}{2}}(D_R)} \sup_{n \in \N} \|v_n\|_{L^{p'}(\R^2)}^2 \to 0,
  	\text{ as }R \to \infty, 
  	\end{equation}
  	since $2<p<\infty$.
  	Next, decomposing $\R^2$ into disjoint squares $\{Q_\ell\}_{\ell\in\N}$ 
  	of side length $R$, and considering for each $\ell$ the square $Q'_\ell$ 
  	with the same center as $Q_\ell$ but with side length $3R$, we obtain by an estimate
  	similar to \cite[pp. 109-110]{alama-li92},
  	\begin{align*}
  	&\Bigl|\int_{\R^2} v_n [(1_{A_R} \Phi_2)*v_n]\,dx \Bigr|
  	\leq \sum_{\ell=1}^\infty\int_{Q_\ell}\Bigl(\int_{\frac1R<|x-y|<R}|\Phi_2(x-y)\,||v_n(x)|\, |v_n(y)|\, dy\Bigr) dx\\
  	&\qquad\leq
  	C R^{\frac4p}(1+\log R)
  	\sum_{\ell=1}^\infty\Bigl(\int_{Q'_\ell}|v_n(x)|^{p'}\,dx\Bigr)^{\frac{2}{p'}}\\
  	&\qquad\leq C R^{\frac4p}(1+\log R) 
  	\Bigl[\:\sup_{\ell \in \N}\int_{Q'_\ell}|v_n(x)|^{p'}\, dx\Bigr]^{\frac{2}{p'}-1} 
  	\:\sum_{\ell=1}^\infty  \int_{Q'_\ell}|v_n(x)|^{p'}\,dx\\ 
  	&\qquad\leq C R^{\frac4p}(1+\log R) 
  	\Bigl[\:\sup_{y \in \R^2}\int_{B_{3R\sqrt{2}}(y)}|v_n(x)|^{p'}\,  dx\Bigr]^{\frac{2}{p'}-1}
  	\: \|v_n\|_{p'}^{p'}. 
  	\end{align*}
  	The assumption \eqref{eqn:vanish_phi2} therefore gives
  	$\int_{\R^2} v_n [(1_{A_R} \Phi_2)*v_n]\,dx\to 0$, as $n\to\infty$,
  	for every $R>0$. Combining this with \eqref{eq:14}, the claim follows.

  	{\bf Claim 2}: 	Let $6<p\leq\infty$, $\lambda_p:=\frac12 -\frac3p>0$ and 
  	$M_R:= \R^N \setminus B_R$ for $R>0$. Then there exists a
  	constant $C>0$ such that, for all $R \geq 1$ and $f\in\cS(\R^2)$ with 
  	$\supp \wh{f} \subset \{\xi \::\: ||\xi|-1|\le \frac{1}{2}\}$,  
  	$$
  	\|[1_{M_R} \Phi_1 ] \ast f\|_p\leq C R^{-\lambda_p}\|f\|_{p'}.
  	$$ 
  	
  	Since $\Phi_1$ is bounded, it suffices to prove the assertion for $R\geq 4$. 
  	For this, let us replace in the decomposition \eqref{eqn:dyadic_Q}
  	the function $\Phi_1$ by $P_R:=1_{M_R}\Phi_1$. Then,
  	\begin{equation*}
  	P_R\ast\varphi= \sum \limits_{j=[\log_2 R] }^\infty Q^j \qquad \text{with 
  		$Q^j:= (P_R \phi_j)\ast\varphi$ for $j \in \N$,} 
  	\end{equation*}
  	using the fact that $P_R\phi_j=0$ for all $j$ such that $2^{j+1}\leq R$.
  	Since the asymptotic behavior of $\Phi_1$ and $P_R$ are identical,
  	the arguments used there give as in \eqref{eqn:stein_tomas_phi1},
  	$$
  	\|Q^j\ast f\|_2\leq C 2^{\frac{j}2}\|f\|_{\frac65}\quad\text{ for all }f\in\cS(\R^2),\ j\geq [\log_2R]
  	$$
  	with a constant $C>0$ independent of $j$ and $R$. By duality and interpolation
  	\begin{equation*}
  	\|Q^j \ast f\|_3 \leq C \, 2^{\frac{j}{2}} \|f\|_{\frac43}\quad\text{for all }
  	f\in\cS(\R^2),\ j\geq [\log_2 R].
  	\end{equation*} 
  	Interpolating this last estimate with the $L^1$--$L^\infty$ estimate
  	$$
  	\|Q^j \ast f\|_\infty \leq C 2^{-\frac{j}{2}}  \|f\|_1,\quad f\in\cS(\R^2),\ j\geq [\log_2 R],
  	$$
  	similar to \eqref{eqn:l1_linf_phi1}, we obtain for $p\geq 3$,
  	$$
  	\|Q^j \ast f\|_p \leq C \, 2^{j\left(\frac3p-\frac12\right)}\|f\|_{p'} 
  	= C \, 2^{-j\lambda_p}\|f\|_{p'}\quad \text{for all }f\in\cS(\R^2),\ j\geq [\log_2R].
  	$$
  	Since $P_R \ast f = \frac1{2\pi}(P_R\ast\varphi) \ast f$ for all 
  	$f\in\cS(\R^2)$ with $\supp\wh{f}\subset\{\xi \::\:\bigl||\xi|-1\bigr|\le \frac{1}{2}\}$, 
  	and since $\lambda_p>0$ for $p>6$,
  	we conclude that for all such $f$ and all $p>6$, 
  	$$
  	\|[1_{M_R} \Phi_1 ] \ast f\|_p = \frac1{2\pi}\|(P_R\ast\varphi) \ast f\|_p 
  	\le C \|f\|_{p'} \sum_{j=[\log_2  R]}^\infty 2^{-j\lambda_p} 
  	\le  C R^{-\lambda_p} \|f\|_{p'}.
  	$$
  	The claim is proved.

  	{\bf Claim 3}: Let $6<p \leq\infty$ and suppose that $(v_n)_n\subset \cS(\R^2)$ 
  	is a bounded sequence in $L^{p'}(\R^2)$ such that \eqref{eqn:vanish_phi2}
  	holds. Then $\int_{\R^2} v_n [\Phi_1 * v_n]\,dx \to 0$ as $n \to \infty$.
  	
  	To prove the claim, fix a radial function $\chi\in \cS(\R^2)$ such that
  	$\wh{\chi}\in C^\infty_c(\R^2)$ is radial, $0\leq \wh{\chi}\leq 1$,
  	$\wh{\chi}(\xi)=1$ for $| |\xi|-1|\leq \frac14$ and $\wh{\chi}(\xi)=0$
  	for $| |\xi|-1|\geq\frac12$. Moreover, let $w_n:= \chi \ast v_n$. 
  	We then have $\Phi_1 * v_n = \frac1{2\pi}\Phi_1 * w_n$, since 
  	$\supp \wh{w_n}\subset\{\xi\, :\, ||\xi|-1|\leq\frac12\}$. 
  	Hence, the decomposition
  	$$
  	\int_{\R^2} v_n [\Phi_1 * v_n]\,dx 
  	= \frac1{2\pi}\int_{\R^2} v_n [(1_{B_R} \Phi_1 ) * w_n]\,dx 
  	+ \frac1{2\pi} \int_{\R^2} v_n	[(1_{M_R} \Phi_1 ) * w_n]\,dx
  	$$
  	holds for all $n$ and using Claim 2 we obtain 
  	\begin{equation}\label{eqn:phi1_mr}
  	\sup_{n \in \N}\Bigl|\int_{\R^2} v_n [(1_{M_R} \Phi_1)  * w_n]\,dx\Bigr| \to 0,
  	\quad \text{as } R\to\infty.
  	\end{equation}
  	Moreover, using the disjoint squares $Q_\ell$ and $Q_\ell'$ of the proof of Claim 1 
  	(cf. also \cite[Lemma 4.3]{evequoz-weth-dual}) and the assumption \eqref{eqn:vanish_phi2}, 
  	we obtain
  	$$
  	\lim_{n \to \infty} \int_{\R^2} v_n [(1_{B_R} \Phi_1)*w_n]\,dx= 0 
  	\qquad \text{for every $R>0$.}
  	$$
  	Combining this with \eqref{eqn:phi1_mr}, the claim follows.

  	Arguing by contradiction, we obtain from Claim 1 and Claim 3 the conclusion of the
  	theorem for sequences $(v_n)_n\subset\cS(\R^2)$, bounded in the $L^{p'}$-norm. 
  	Since the bilinear form $\int_{\R^2} v\cR v\, dx$ is continuous on $L^{p'}(\R^2)$
  	and since the property \eqref{eqn:liminf} is stable under approximation
  	in the $L^{p'}(\R^2)$-norm, the conclusion follows by density.
  \end{proof}
  
  We end this paper by giving the proof of the part (b) of our second main result, showing
  the existence of a nontrivial solution pair for \eqref{eqn:nlh_integral}.
  \begin{proof}[Proof of Theorem~\ref{thm:exist}(b)]
  	Consider the set of paths 
  	$\Gamma=\{\gamma\in C([0,1],L^{p'}(\R^2))\, :\, \gamma(0)=0,\, J(\gamma(1))<0\}$ 
  	and the energy level
  	$$
  	c=\inf\limits_{\gamma\in\Gamma}\max\limits_{t\in[0,1]}J(\gamma(t)).
  	$$
  	Notice that $\Gamma\neq\varnothing$ and $c>0$, since $0$ is a strict local minimum
  	of $J$ and there are $v\in L^{p'}(\R^2)$ such that $J(v)<0$. Indeed, the proof
  	of Theorem \ref{thm:exist}(a) gives these facts without any decay assumption on $Q$. 
  	Using the standard deformation lemma (see \cite[Lemma 2.3]{willem}), 
  	we obtain the existence of a Palais-Smale sequence $(v_n)_n\subset L^{p'}(\R^2)$
  	such that $J(v_n)\to c$, as $n\to\infty$. By \eqref{eqn:bnd_PS}, this sequence is bounded and
  	therefore, as $n\to\infty$,
  	$$
  	\int_{\R^2}Q^\frac1pv_n\mR(Q^\frac1pv_n)\, dx
  	=\left(\frac1{p'}-\frac12\right)^{-1}\left(J(v_n)-\frac1{p'}J'(v_n)v_n\right)\to 
  	\frac{2p'c}{2-p'}>0.
  	$$
  	Consequently, Theorem~\ref{thm:nonvanishing}
  	gives $R, \zeta>0$ and $(x_n)_n\subset\R^2$ satisfying \eqref{eqn:liminf}, 
  	up to a subsequence. Making $R$ larger we may assume $x_n\in\Z^2$ for all $n$. 
  	Since $Q$ is $\Z^2$-periodic, the functional $J$ is invariant under $\Z^2$-translations and 
  	setting $w_n:=v_n(\cdot-x_n)$ we find that $(w_n)_n$ is also a bounded Palais-Smale 
  	sequence for $J$. Going to a subsequence, $w_n\weakto w\in L^{p'}(\R^2)$.
  	Moreover, if $\varphi\in C^\infty_c(\R^2)$,
  	\begin{align*}
  	&\left|\int_{\R^2} \left(|w_n|^{p'-2}w_n  -|w_m|^{p'-2}w_m\right)\varphi\, dx\right| \\
  	&\qquad\qquad\qquad\leq\|J'(w_n)-J'(w_m)\|_\ast\|\varphi\|_{p'} +\|1_B\mK (w_n-w_m)\|_p\|\varphi\|_{p'}\to 0,
  	\end{align*}
  	as $m,n\to\infty$, thanks to the compactness of $1_B\mK$, where $B\subset\R^2$ contains $\supp(\varphi)$.
  	Since $C^\infty_c(\R^2)$ is dense in $L^{p'}(\R^2)$, we infer that for all bounded and measurable $B\subset\R^2$,
  	$(1_B|w_n|^{p'-2}w_n)_n$ is a Cauchy sequence in $L^p(\R^2)$ and thus,
  	$$
  	1_B|w_n|^{p'-2}w_n\to 1_B|w|^{p'-2}w\  \text{ as }n\to\infty,\quad\text{ strongly in }L^p(\R^2).
  	$$
  	Recalling \eqref{eqn:liminf}, we see that
  	$$
  	\int_{B_R(0)}|w|^{p'}\, dx=\lim_{n\to\infty}\int_{B_R(0)}|w_n|^{p'}\, dx\geq\zeta>0,
  	$$
  	and consequently, $w\neq 0$. In addition, for all $\varphi\in C^\infty_c(\R^2)$, we obtain
  	\begin{align*}
  	J'(w)\varphi&=\int_{\R^2}|w|^{p'-2}w\varphi -\int_{\R^2}\varphi \mK w\,dx\\
  	& =\lim_{n\to\infty} \int_{\R^2}|w_n|^{p'-2}w_n\varphi -\int_{\R^2}\varphi \mK w_n\,dx
  	= \lim_{n\to\infty}J'(w_n)\varphi=0,
  	\end{align*}
  	and we conclude that $J'(w)=0$. 
  	Hence, $w$ is a nontrivial critical point of $J$ and, applying Lemma~\ref{lem:regularity},
  	the theorem follows.
  \end{proof}

\section*{Acknowledgements}
This research is supported by the Grant WE 2821/5-1 of the Deutsche Forschungsgemeinschaft (DFG).
The author would like to thank Tobias Weth for suggesting the study of the 2-dimensional case and
for his valuable remarks on a preliminary version of the manuscript.

\bibliographystyle{abbrv}
\bibliography{literatur.bib}

\end{document}